\documentclass[11pt,twoside,reqno,centertags,english]{amsart}
\usepackage{lettrine}

\usepackage[utf8]{inputenc}
\usepackage[english]{babel}
\usepackage{csquotes}
\usepackage{url}

\usepackage{bbding}
\usepackage{hyperref}
\usepackage{geometry}
\hypersetup{
colorlinks,
    citecolor=blue,
    filecolor=blue,
    linkcolor=blue,
    urlcolor=blue
    }
\def\be{\begin{equation}}
\def\ee{\end{equation}}
\newcommand {\da} {d \mathcal{H}^{m-1}}
\newcommand {\p} {\partial}
\usepackage{amsmath, amsfonts, amssymb}
\usepackage{amsthm}
\newtheorem{theorem}{Theorem}[section]
\newtheorem{lemma}[theorem]{Lemma}

\newtheorem{corollary}[theorem]{Corollary}

\theoremstyle{remark}
\newtheorem{remark}[theorem]{Remark}

\theoremstyle{definition}
\newtheorem{definition}[theorem]{Definition}

\usepackage{mathtools}

\newcommand{\beq}{\begin{equation}}
	\newcommand{\eeq}{\end{equation}}
\newcommand{\beqs}{\begin{eqnarray}}
	\newcommand{\eeqs}{\end{eqnarray}}
\newcommand{\beql}{\begin{equation} \label}
	\newcommand{\half}{\frac{1}{2}}

	\newcommand{\calC}{{\mathcal{C}}}

	\newcommand{\calF}{{\mathcal{F}}}

	\newcommand{\curl}{\mathop{\rm curl}\nolimits}

	\newcommand{\veps}{\varepsilon}

	\newcommand{\dee}{\mathcal{D}}
	\newcommand{\scl}{\mathcal{L}}
	
	\newcommand{\sco}{\mathcal{O}}
	\newcommand{\s}{\mathsf{s}}

\makeatletter
\DeclareRobustCommand\widecheck[1]{{\mathpalette\@widecheck{#1}}}
\def\@widecheck#1#2{%
    \setbox\z@\hbox{\m@th$#1#2$}%
    \setbox\tw@\hbox{\m@th$#1%
       \widehat{%
          \vrule\@width\z@\@height\ht\z@
          \vrule\@height\z@\@width\wd\z@}$}%
    \dp\tw@-\ht\z@
    \@tempdima\ht\z@ \advance\@tempdima2\ht\tw@ \divide\@tempdima\thr@@
    \setbox\tw@\hbox{%
       \raise\@tempdima\hbox{\scalebox{1}[-1]{\lower\@tempdima\box
\tw@}}}%
    {\ooalign{\box\tw@ \cr \box\z@}}}
\makeatother

\DeclareMathOperator{\tr}{Tr}

\DeclareMathOperator{\di}{div}

\newcommand {\R} {\mathbb{R}}

\newcommand{\vbs}{\bar{v}} 
\newcommand{\cc}{\mathfrak C(\vbs,v_0)}
\newcommand{\vv}{\mathfrak U}
\newcommand{\ren}{K_{v,\vbs}} 
\newcommand{\bigk}{\mathcal K(E-\vbs,B)}
\newcommand{\bigkm}{\mathcal K(E_m-\vbs,B_m)}
\newcommand{\ev}{Q}
\newcommand{\iii}{\Pi}
\newcommand{\cci}{\mathfrak C(\iii \vbs,\iii v_0)}
\newcommand{\bigkold}{\mathcal K(\ev,B)}
\newcommand{\bigkp}{\mathcal K(E_+-\vbs,B_+)}
\numberwithin{equation}{section}
\newcommand{\ach}[1]{{#1}}
\newcommand{\dv}[1]{{#1}}
\date{}

\title[Nash system]{
On the variational dual formulation of the Nash system and an adaptive convex gradient-flow approach to nonlinear PDEs
}
\author[D.~Vorotnikov]{Dmitry Vorotnikov}
\address[D.~Vorotnikov]{University of Coimbra, CMUC, Department of Mathematics,  3000-143 Coimbra, Portugal}{}
\email{mitvorot@mat.uc.pt}
\author[A.~Acharya]{Amit Acharya}
\address[A.~Acharya]{Department of Civil \& Environmental Engineering and Center for Nonlinear Analysis, Carnegie Mellon University, Pittsburgh, 15213, PA, USA}{}
\email{acharyaamit@cmu.edu}

\begin{document}

	\begin{abstract} 
		We investigate the influence of base states on the consistency of the dual variational formulation for quadratic systems of PDEs, which are not necessarily conservative (typical examples include the noise-free Nash system with a quadratic Hamiltonian and multiple players). We identify a sufficient condition under which consistency holds over large time intervals. In particular, in the single-player case, there exists a sequence of base states (each exhibiting full consistency) that converges in mean to zero. We also prove existence of variational dual solutions to the noise-free Nash system for arbitrary base states. Furthermore, we propose a scheme based on Hilbertian gradient flows that, starting from an arbitrary base state, generates a sequence of new base states that is expected to converge to a solution of the original PDE.
	\end{abstract}
	
	\maketitle
	
	
	\textbf{MSC [2020] 35D99, 35L40, 37K58, 41A60, 49Q99}

	\section{Introduction}
	
	\subsection{Nash system} Consider the noise-free Nash system with the quadratic Hamiltonian and $N$ players, cf. \cite{Hy21}:  
	\be \label{e:hje1} -\partial_\tau \psi_i+\frac 1 2 |\partial_{x_i}\psi_i|^2+\sum_{j\neq i}\partial_{x_j}\psi_j \partial_{x_j}\psi_i=0 ,\ i=1,\dots, N.\ee 
	Here\footnote{The Einstein summation convention is not used here but will be adopted in Section~\ref{sec:dual_grad_flow_setting}.} the unknown function is $\psi=\psi(\tau,x)$, where $\tau \in [0,T]$ and $x$ belongs to the periodic box $\mathbb T^N$. The system is complemented with the terminal condition \be\psi(T,x)=\psi_*(x).\ee
	More generally, if the states of the players are $p$-dimensional, we have the following system, cf. \cite{CLions}: 
	
	\be \label{e:hjed}  -\partial_\tau \psi_i+\sum_{l=1}^p\left[\frac 1 2 |\partial_{x_{il}}\psi_{i}|^2+\sum_{j\neq i}\partial_{x_{jl}}\psi_j \partial_{x_{jl}}\psi_i\right]=0 ,\ i=1,\dots, N.\ee
	Here $\tau \in [0,T]$ and $x$ is in the periodic box $\mathbb T^{N \times p}$.

	Then \eqref{e:hjed} --- as well as its particular case \eqref{e:hje1} --- can be written in the form
	\be -\partial_\tau \psi+\vv (\nabla \psi \otimes \nabla \psi)=0, \ \psi(T,x)=\psi_*(x),\ee where $\vv:\mathbb R^{n\times n}_s\to \mathbb R^{N}$ is a finite-dimensional linear operator (hereafter $n=pN^2$). 
	More explicitly, the elements of $\R^{n\times n}_s$ can be represented in the form $A=\{A_{ijk,hqr}\}$, where $i,j,h$ and $q$ vary between $1$ and $N$, whereas $k$ and $r$  vary between $1$ and $p$; then $\vv$ acts according to \be \label{e:opu} (\vv A)_i=\frac 1 2  \sum_{l=1}^d\left[A_{iil,iil}+\sum_{j\neq i}(A_{jjl,ijl}+A_{ijl,jjl})\right].\ee
	
	Setting $t:=T-\tau$, and slightly abusing the notation by letting $\psi(t,x):=\psi(\tau,x)$, we rewrite the problem in the form 
	\be \label{e:mainhj} \partial_t \psi+\vv (\nabla \psi \otimes \nabla \psi)=0, \ \psi(0,x)=\psi_*(x).\ee
	
	The particular case $N=1$ and $\vv=\frac 1 2 \tr$ corresponds to the classical quadratic Hamilton-Jacobi (H-J) equation.
	
	To the best of our knowledge, \textbf{no solvability results are currently available for the noise-free Nash system} with 
	$N
	>
	1$, cf. \cite{Hy21}. This absence is related to the lack of a comparison principle for H–J systems. In particular, there is no analogue of viscosity solutions for the noise-free Nash system. Weak solutions in the sense of distributions can still be defined (cf. Definition \ref{d:ws}), though, of course, without any expectation of uniqueness. Furthermore, due to the lack of compactness, no existence result for such solutions is currently known. Finally, no selection criterion has been established to address the non-uniqueness phenomenon, although vanishing viscosity may provide a possible approach.
	
	Beyond the natural applications to game theory of the H-J systems, another primary motivation  of our study comes from the modern theory of mechanics of defects in continuum mechanics, with many outstanding classical as well as cutting-edge applications in engineering and science. The following set of PDEs is a `simplest' model for the dissipative dynamics of localized, non-singular dislocation line defects and cracks in elastic solids and geophysical rupture \cite{zhang2015single, arora_1,zhang2017non, morin2021analysis, acharya2024coupled}: given a domain $\Omega \subset \R^3$, $0 \leq \veps \ll 1$, and a multi-well, non-convex $f: \R^{3 \times 3} \to \R$, we look for $P:\Omega  \times [0,T] \to \R^{3 \times 3}$, $\theta:\Omega \times [0,T] \to \R^3$ that solve\footnote{Notation: $\left(X[AB]\right)_i = \sum_{j,k,M}\epsilon_{ijk} A_{jM}B_{Mk}$; $(\curl P)_{Mi} = \sum_{j,k}\epsilon_{ijk} \p_j P_{Mk}$, where $\epsilon_{ijk}$ is the Levi-Civita symbol.} 
	\begin{equation*}
		\begin{aligned}
			\p_t P & = - \operatorname{curl} P \times X \left[ \Big( - (\nabla \theta - P) + \veps \operatorname{curl}\operatorname{curl} P + \p_P f(P) \Big)^T \operatorname{curl} P \right],\\\p_{tt} \theta & = \operatorname{div}(\nabla \theta - P),
		\end{aligned}
	\end{equation*} 
	with appropriate initial and boundary conditions (see \cite{ach2} and earlier references therein for a model without simplifying assumptions). Particular physically meaningful ans\"atze often allow assuming $P$ to have $N$ non-zero components, $N \leq 9$, possibly varying in $d \leq 3$ essential directions in space (similar reductions can be implemented for the field $\theta$ as well). For $N = 1$, $d = 1,2$, robust computational techniques have been developed for this set of equations generating useful physical insights into defect dynamics in elastic solids \cite{zhang2015single}, cracks \cite{morin2021analysis}, nematic liquid crystals \cite{zhang2017non}, and geophysical rupture dynamics \cite{zhang2015single, arora_1}. For $N > 1$ this is a H-J-like system,  
	and, to our knowledge, no theory or numerical schemes are available for it --- see \cite{acharya2023vector} for a first analysis in the static and \cite{AT11,AS23} in the dynamic settings. The dual methodology studied in this paper marks the beginning of a rigorous analysis of such dynamical models; for a model and analysis of plasticity from dislocations in the setting of integral currents, see \cite{Rind25}, and within a phase-field setting see \cite{Mon_etal_06} as a representative of the work of Monneau and co-workers related to dislocations and H-J type equations. In the practical realm, the variational structure and the gradient flow scheme that we will propose already has the significant advantage of suggesting a natural finite dimensional Rayleigh-Ritz discretization for a multidimensional system involving nonlinear transport which does not easily, if at all, lend itself to discretization based on existing numerical algorithms.
	
	\subsection{Variational dual formulations: an overview} In this paper, we are concerned with the variational dual formulation of such nonlinear systems, with an emphasis on the noise-free Nash system. The idea originates with Brenier \cite{CMP18, Br20}, who proposed, for the incompressible Euler equations and related models, to study solutions that minimize the Lagrangian action, i.e., the time integral of the kinetic energy (for recent applications of the least action principle as a selection criterion for systems of hyperbolic conservation laws, see \cite{gimperlein_etal_24, gimperlein_etal_25}). Although this minimization problem may fail to have a solution, it leads to a dual problem with more favorable convexity properties. Brenier derived an explicit relation linking smooth solutions of the incompressible Euler equations on short time intervals with the \emph{variational dual solutions} (\ach{a terminology introduced in \cite{sga})}. In \cite{BMO22}, this technique was applied to the multi-stream Euler-Poisson system. A numerical implementation for the dual variational formulation of the quadratic porous medium equation and the Burgers equation has recently been done in \cite{mirebeau_stampfli}. In \cite{V22, V25, V26b}, the first author extended Brenier's approach by identifying structures in nonlinear PDEs that allow variational dual formulations with favorable properties. These dual problems are related to optimal transport, specifically its \emph{ballistic} variant \cite{V25, mirebeau_stampfli}. Notably, the functionals arising in this context are not expected to be self-dual, so the formalism of \cite{ghoussoub2009} is not applicable.

	A similar variational dual formulation has been proposed by one of us in \cite{acharyaQAM,ach2}, and applied to the study of nonlinear PDEs in \cite{ach2,sga,ASZ24,kpa,Ach6,AGS24,AG25} and the references therein. A key feature of this approach is that the dual problem is built around a \emph{base state}, cf. Section \ref{sbs}, which can be seen as an ``initial guess'' for the solution of a PDE in question; accordingly, the kinetic energy is replaced with an appropriate auxiliary potential with positive-definite Hessian,  e.g., with the relative kinetic energy.

	The primary goal of these approaches is to define convex variational principles for conservative and non-conservative (including dissipative) physical models to facilitate their solution, whether exact or approximate (through computation). For instance, they allow, using methods of the calculus of variations, the definition and proofs of existence of variational dual solutions to the Euler-Lagrange PDE systems of physical energy functionals whose global minimizers do not exist \cite[Sec.~5]{sga}, \cite{AGS24}.  
	
	\subsection{Base states and consistency} \dv{\label{sbs} Throughout the paper, we repeatedly invoke two concepts. Rather than fixing a single formal mathematical definition, we allow their rigorous formulation to depend on the context in which they arise. They provide a framework for formulating and interpreting the results. Let us describe them heuristically. }
	
	\medskip
	
	\noindent\emph{Concept I (Base state).} \dv{The key task in the variational approach with which we are concerned is to find the critical points of a certain {\it dual }functional that underlies the minimax formulation of the dual problem. This functional is constructed by considering the equation of interest,  the {\it primal} PDE system, in a standard weak form with the place of test functions taken by ``dual'' (Lagrange multiplier) fields, along with an additional auxiliary potential which is open to design in order to facilitate the solution or approximation of the given primal system. In particular, the design of the auxiliary potential admits  the inclusion of specified functions of space and time, referred to as \emph{base states} \cite[Sec.~5]{ach2}, in order to guide solutions of the variational problem to recover those of the primal system. Formally, every critical point of any dual functional generates a solution of the primal problem, independently of the choice of the auxiliary potential. As an example of such design of the auxiliary potential,  it suffices, for equations with quadratic nonlinearities like the Nash system, to choose a quadratic functional of the difference between the primal fields and the corresponding base states, referred to as the \emph{relative energy}.  The intuition behind the notion of a base state is that it can be viewed as a preliminary approximation to the solution, possibly motivated by heuristic or other considerations. However, this interpretation is not binding. In principle, any function belonging to the same regularity class as the expected solution of the PDE under consideration (the base state may, in fact, be more regular than the solution) can be chosen as the base state. From this perspective, the base states in \cite{CMP18, Br20, V22, V25,V26b,mirebeau_stampfli} are simply zero.}

	\medskip
	
	\noindent\emph{Concept II (Consistency).} 
	\dv{After fixing a time interval and a base state, the dual problem is constructed through a suitable minimax procedure. We say that the dual formulation is \emph{consistent} if there exists a rigorous systematic correspondence between solutions of the original PDE and solutions of the dual problem. This correspondence may be realized either via the DtP map, cf. Section~\ref{sec:var_set} (see also \cite[p. 587, after (2.22)]{CMP18}), or through tailored constructions specific to conservative evolutionary systems with the base state chosen to be zero, cf. \cite{V22,V25,V26b}; see also \cite[Remark~4.5]{V22} for a comparison of the two approaches.   Although solutions of the original PDE may fail to exist for some initial data, such a correspondence should apply at least for a suitable class of initial data for which solutions do exist. A somewhat unsettling aspect of the duality scheme with the zero base state is that many \emph{consistency} results are obtained only on small time intervals, cf. \cite[Theorems 2.3 and 3.1]{CMP18}, \cite[Theorem 3.8]{V22}, \cite[Theorem 5.1]{BMO22}. In \cite{V25,V26b}, the first author succeeded in extending the consistency for conservative systems (still assuming zero base states) to large time intervals by introducing suitable time-dependent weights; as an application, he derived a variant of Dafermos principle for various PDE models that admit the variational dual formulation. On the other hand, even without time-dependent weights, the inclusion of non-zero base states in the dual variational formulation can provide global-in-time consistency. Indeed, it was recently shown in \cite{ASZ24} that, for every weak solution of the incompressible Euler or Navier-Stokes equations, there exists at least one base state and an associated dual functional with this property. In the present paper, we show that such base states are, in fact, abundant.}
	
	\medskip

	\subsection{Goals and organization of the paper} In Section \ref{s:se2}, we investigate how the deployment of base states (without any time-dependent weights) influences consistency. We identify a broad
	class of base states that ensure consistency of the dual formulation of the noise-free Nash system over any interval $[0,T]$ (Theorem \ref{t:smooth}). As a consequence, in the single-player setting ($N=1$), we build a sequence of base states --- each exhibiting full consistency --- that converges in mean to zero (Corollary \ref{c:cor}). Moreover, we show that, for 
	$N>1$, a variational dual solution to the noise-free Nash system exists and belongs to a certain $L^2$-space, without any restrictions on base states and initial data (Theorem \ref{t:ex}). 
	
	Subsequently, in Section \ref{sec:dual_grad_flow_setting} we propose a novel scheme, based on Hilbertian gradient flows and suitable for a wide class of nonlinear PDEs, that produces, from an initial base state, a sequence of new base states anticipated to converge to a solution of the original PDE. The proposed adaptive gradient flow scheme is a natural outgrowth of the practical experience with computational schemes based on Newton's method for approximating the Euler-Lagrange system of the dual variational principle. From the very first nonlinear problem solved in \cite{ka1} (Euler's ODE system for dynamics of a rigid body), it was clear that the idea of base states was crucial to the practical success of the scheme, further confirmed in \cite{ka2} in the context of (inviscid) Burgers equation (see also \cite{mirebeau_stampfli}) and in \cite{sga} for nonconvex elastostatics and dynamics in space dimension $1$). In \cite{kpa}, base state changes with step-size control in a Newton approximation scheme were used, and the possibility of a gradient flow-based approach was alluded to. The gradient flow scheme formalized here in detail constitutes a new and logically natural development of these ideas. We work in an exact setting, i.e., without approximation.
	
	Our gradient flow scheme proceeds through several stages of gradient descent in the Hilbert space of functions of $t$ and $x$, with respect to a fictitious, time-like variable 
	$\s$ along which the evolution unfolds. Transitions between stages correspond to the selection of a new base state, unambiguously determined by the scheme, without disrupting the continuity of the overall trajectory.    In Section \ref{s:toy}, using a simple toy model, we demonstrate how the switching of base states, the convergence of the scheme, and the eventual generation of a solution occur in practice. Section \ref{sec5} presents our scheme in the particular case of the noise-free Nash system.

	\section{The dual variational formulation of the noise-free Nash system. Consistency and existence of solutions.} \label{s:se2}
	
	\subsection*{Basic notation} \dv{Throughout Sections \ref{s:se2} and \ref{sec5}, $\Omega$ denotes the periodic box $\mathbb T^{N \times p}$. For brevity, write $X:=L^2(\Omega)$. We recall that $n:=pN^2$}. We use the notations $\R^{n\times n}$ and $\R^{n\times n}_s$ for the spaces of $n\times n$ matrices and symmetric matrices, resp., with the scalar product generated by the Frobenius norm.  Let $X^{n\times n}_s$ be the subspace of $X^{n\times n}$ consisting of symmetric-matrix-valued functions. The parentheses $(\cdot,\cdot)$ will stand for the scalar products in $X^n$ and $X^{n\times n}_s$. For $A,B\in X^{n\times n}_s$, we write $A\geq B$ and $A>B$ when $A-B$ is a positive-semidefinite-matrix-function and a positive-definite-matrix-function, resp. The symbol $I$ stands for the identity matrix of a relevant size. Denote by $\iii:X^n \to X^n$ the averaging operator defined by $$\iii v(x)=\int_\Omega v(y)\,dy.$$
	
	\subsection{The variational dual formulation of the Nash system} \label{sec21} We begin by presenting the variational dual formulation of the noise-free Nash system, following the approach outlined in \cite{CMP18, V22} and employing the base states introduced in \cite{ach2}.
	
	We set $v:=\nabla \phi$, $v_0:=\nabla \phi_*$, and define the differential operator $L$ by \be \label{e:opl} Lw=-\nabla (\vv w).\ee \dv{Here $w: \Omega \to \mathbb R^{n\times n}_s$ is any sufficiently regular function. By construction, $Lw:\Omega \to \mathbb R^n$}. 
	
	\dv{We argue that} a natural (Burgers-like) weak formulation of our problem \eqref{e:mainhj} consists in finding a function $v:[0,T]\to X^n$ such that \be \label{e:w1}\int_0^T \left[(v-v_0,\p_t a)+(v\otimes v, L^*a)\right]\, dt=0\ee for all sufficiently smooth vector fields $a: [0,T]\to X^n$, $a(T)= 0$. Here $$L^*=\vv^* \di$$ is the \dv{differential operator adjoint to $L$}. \textbf{Note that this weak formulation includes the information that $v$ is a gradient and the initial condition}. Indeed, multiplying   \eqref{e:mainhj} by $\di a$, integrating w.r.t. $(0,T)\times \Omega$ and observing that \be \label{e:auxid} \int_0^T (\p_t \psi, \di a)\,dt=-\int_0^T (\p_t v, a)\,dt=\int_0^T (v-v_0,\p_t a)\,dt \ee we get \eqref{e:w1}. Conversely, if $b$ is any smooth divergence-free (w.r.t. $x$) field on $[0,T]\times \Omega$, test \eqref{e:w1} with $a(t)=\int_T^t b(s)\,ds$ to obtain $\int_0^T (v-v_0,b)\, dt=0$, which yields that $v-v_0$ is a gradient, and hence $v=\nabla \psi$ for some $\psi:[0,T]\times \Omega \to \R^n$ that is determined up to adding a function of $t$. With this at hand, and $a$ being now chosen arbitrary within its range, we observe that \eqref{e:w1} implies (up to regularity issues) that $$(\psi(0)-\psi_*, \di a(0))+\int_0^T (\p_t \psi+\vv (\nabla \psi \otimes \nabla \psi), \di a)\,dt=0,$$ which means that there exists a function $g(t)$ and a constant $c$ such that $\partial_t \psi+\vv (\nabla \psi \otimes \nabla \psi)=g(t)$ and $\psi(0)=\psi_*+c$.  Since $\psi$ has been defined up to adding a function of time, without loss of generality we may assume that $g(t)\equiv 0$ and $c=0$ (otherwise replace $\psi$ with $\psi-c-\int_0^t g(s) \,ds$). 
	
	Let us now rewrite problem \eqref{e:w1} in terms of the test functions $B:=L^*a$ and $E:=\p_t a$. The link between $B$ and $E$ can alternatively be described by the condition \be \label{e:constrweak}\int_0^T \left[(B,\p_t \Psi)+(E, L \Psi)\right]\, dt=0\ee for all smooth vector fields $\Psi: [0,T]\to X^{n\times n}_s$, $\Psi(0)= 0$, cf. \cite{V22}.   In particular, this implies that $B(t,x)$ a.e. belongs to the range of the operator $L^*$.  Note that \eqref{e:w1} becomes \be \label{e:w2}\int_0^T \left[(v-v_0,E)+(v\otimes v, B)\right]\, dt=0.\ee 

	Motivated by this discussion, we adopt the following definition.
	\begin{definition}[Weak solutions] Let  $v_0\in X^n$. A function $v\in L^2((0,T)\times \Omega;\R^n)$ is a \emph{weak solution} to \eqref{e:mainhj} (more precisely, to its Burgers-like formulation) if it satisfies \eqref{e:w2} for all pairs \be \label{e:be} (E,B)\in L^2((0,T)\times \Omega;\R^n)\times L^\infty((0,T)\times \Omega;\R^{n\times n}_s)\ee meeting the constraint \eqref{e:constrweak}. \label{d:ws} \end{definition}
	
	\dv{Following \cite{ach2,ach3}, we fix a \emph{base state} $\vbs\in L^2((0,T)\times \Omega;\R^n)$, as anticipated in Section \ref{sbs}. Consider the problem of finding a weak solution to \eqref{e:mainhj} that minimizes the time integral of the relative energy \begin{equation} \label{e:consp1} \ren(t):=\frac 12(v(t)-\vbs(t),v(t)-\vbs(t)).\end{equation} (We do not address the question of whether the set of weak solutions is nonempty). This task} can be implemented via the saddle-point problem  \be\label{e:sadd1}\mathcal I(v_0, \vbs, T)=\inf_{v}\sup_{E,B:\,\eqref{e:constrweak}}\int_0^T \left[\ren(t)+(v-v_0,E)+(v\otimes v, B)\right]\, dt.\ee (If a weak solution does not exist, then $\mathcal I(v_0, \vbs, T)=+\infty$). Equivalently, \be\label{e:sadd1+}\mathcal I(v_0, \vbs, T)=\inf_{v}\sup_{E,B:\,\eqref{e:constrweak}}\int_0^T \left[(v-v_0,E-\vbs)+\frac 1 2(v\otimes v, I+2B)\right]\, dt+\cc,\ee where \be \cc:=\int_0^T \left[-(v_0,\vbs)+\frac 1 2(\vbs,\vbs)\right]\, dt.\label{cccd}\ee The infimum in \eqref{e:sadd1} is taken over all $v\in L^2((0,T)\times \Omega;\R^n)$, and the supremum is taken over all pairs $(E,B)$ satisfying \eqref{e:be} and the linear constraint \eqref{e:constrweak}.
	
	The dual problem is \be\label{e:sadd2}\mathcal J(v_0, \vbs, T)=\sup_{E,B:\,\eqref{e:constrweak}}\inf_{v}\int_0^T \left[(v-v_0,E-\vbs)+\frac 1 2(v\otimes v, I+2B)\right]\, dt+\cc,\ee where $v,E,B$ are varying in the same function spaces as above.  
	
	It is easy to see that any solution to \eqref{e:sadd2} necessarily satisfies \be\label{e:bwe} I+2B\geq 0\ \mathrm{a.e.}\ \mathrm{in}\ (0,T)\times \Omega.\ee Following \cite{CMP18,V22}, consider the nonlinear functional $$\mathcal K: L^2((0,T)\times \Omega;\R^n)\times L^\infty((0,T)\times \Omega;\R^{n\times n}_s)\to \R$$ defined by the formula
	\be\label{e:defk1} \bigkold=\inf_{z\otimes z\le M}\int_0^T \left[(z,\ev)+\frac 1 2(M, I+2B)\right]\, dt,\ee  where the infimum is taken over all pairs $(z,M)\in L^2((0,T)\times \Omega;\R^n)\times L^1((0,T)\times \Omega;\R^{n\times n}_s)$. Note that $\bigkold$ is a negative-semidefinite quadratic form w.r.t. $\ev$ for fixed $B$. Mimicking \cite[Remark 3.6]{V22}, it is possible to see that \eqref{e:sadd2} is equivalent to \be\label{e:conc}\mathcal J(v_0, \vbs, T)=\sup_{E,B:\,\eqref{e:constrweak},\eqref{e:bwe}}\int_0^T \left[-(v_0,E)+\frac 1 2 (\vbs,\vbs)\right]\, dt+\bigk,\ee where the supremum is taken over all pairs $(E,B)$ belonging to the class \eqref{e:be}.
	
	\begin{remark} For the sake of generality, in what follows $v_0$ and $\vbs$ are not necessarily gradients because the dual problem is well-defined even without this restriction. \end{remark}

	\subsection{Consistency and existence of solutions}
	
	In this section, we first show that a broad class of base states ensures \emph{consistency} of the dual formulation over any interval $[0,T]$. Specifically, any matrix-valued positive semidefinite ``density'' $G(t,x)$ and the ``velocity field'' $u(t,x)$ solving the generalized “transport” equation \be \p_t G+2 \vv^* \di (Gu) =0, \quad G(T,\cdot)\equiv I \label{e:gte1}\ee define a base state \dv{for which there is no duality gap (in other words, the optimal values of the primal and dual problems, \eqref{e:sadd1} and \eqref{e:sadd2}, coincide). Moreover, for those base states, the solutions of the original problem \eqref{e:mainhj} are in explicit correspondence with the variational dual solutions.} As a corollary, in the single-player case ($N=1$), we construct a sequence
	of base states (each exhibiting full consistency and a one-to-one correspondence between the ``primal'' and ``dual'' solutions) that converges in mean to zero. Finally, we establish the existence of a variational dual solution to the noise-free Nash system for any base state $\bar v$ and $N > 1$.
	
	\begin{remark} Although our main focus is on the noise-free Nash system, we are confident that the results below can be extended to other quadratic PDEs, for instance, to those studied in \cite{CMP18,V22}. \end{remark}
	
	\begin{theorem}[Consistency] \label{t:smooth} Let $v_0\in X^n$ and $v$ be a weak solution to \eqref{e:mainhj}. Select any pair $(G,u)\in L^\infty((0,T)\times \Omega;\R^{n\times n}_s)\times L^2((0,T)\times \Omega;\R^n)$ satisfying the constraints \be\label{e:pd++g} G\geq 0\ \textrm{a.e.}\ \textrm{in}\ (0,T)\times \Omega\ee  and \be \label{e:constrweakg}\int_0^T \left[(G-I,\p_t \Psi)-2(Gu, L \Psi)\right]\, dt=0\ee for all smooth vector fields $\Psi: [0,T]\to X^{n\times n}_s$, $\Psi(0)= 0$. Assume also that \be \label{e:goodbs} \vbs=G(v-u)\ \textrm{a.e.}\ \textrm{in}\ (0,T)\times \Omega.\ee  Then $\mathcal I(v_0, \vbs, T)=\mathcal J(v_0, \vbs, T)$, and the pair $(E_+,B_+)$ defined by \be \label{e:dualsol} E_+:=-Gv+\vbs,\ B_+:=\frac 1 2 (G-I)\ee belongs to the class \eqref{e:be} and maximizes \eqref{e:conc}. In particular, if $G> 0$ a.e. in an open subset $\mathcal O \subset (0,T)\times \Omega$, we can  retrieve the weak solution $v$ (inside $\mathcal O$)  from the variational dual solution $(E_+,B_+)$  by applying the formula \be \label{e:dualsolinv} v=(I+2B_+)^{-1}(\vbs-E_+)\ \textrm{a.e.}\ \textrm{in}\ \mathcal O. \ee  \end{theorem}

	\begin{proof} By construction, $\vbs \in L^2((0,T)\times \Omega;\R^n)$. Consequently, the pair $(E_+,B_+)$ belongs to $L^2((0,T)\times \Omega;\R^n)\times L^\infty((0,T)\times \Omega;\R^{n\times n}_s)$. Moreover, it follows from \eqref{e:constrweakg} and \eqref{e:goodbs} that the pair $(E_+,B_+)$  \dv{satisfies  \eqref{e:constrweak}. In addition, since $I+2B_+=G$, \eqref{e:pd++g} implies \eqref{e:bwe} for $B_+$}.

		By \eqref{e:w2} and \eqref{e:sadd1}, $$\mathcal I(v_0, \vbs, T)\leq \int_0^T \ren(t)\, dt.$$ In view of \eqref{e:conc}, in order to prove that there is no duality gap --- and that $(E_+,B_+)$ is a variational dual solution --- it suffices to show that \be \label{e:claim1} \int_0^T -(v_0,E_+)\, dt+\bigkp \ge \int_0^T \left[\ren(t)-\frac 1 2 (\vbs,\vbs)\right]\, dt.\ee 
		But since $v$, in particular, satisfies \eqref{e:w2} with the test functions $(E_+,B_+)$, we have \be \label{e:w2+}\int_0^T \left[(v-v_0,E_+)+(v\otimes v, B_+)\right]\, dt=0.\ee Hence, by the first  definition in \eqref{e:dualsol}, \be \int_0^T \left[-(v_0,E_+)+(v\otimes v, B_+)\right]\, dt=\int_0^T \left[(v\otimes v, G)-(v,\vbs)\right]\, dt,\ee which in view of the second definition in \eqref{e:dualsol} gives  \be \label{e:w3+}-\int_0^T \left[(v_0,E_+)+(v\otimes v, B_+)\right]\, dt=\int_0^T \left[(v\otimes v, I)-(v,\vbs)\right]\, dt.\ee
		Consequently,  \begin{multline*} \int_0^T -(v_0,E_+)\, dt+\bigkp\\=\int_0^T -(v_0,E_+)\, dt+\inf_{z\otimes z\le M}\int_0^T \left[(z,-Gv)+\frac 1 2(M, I+2B_+)\right]\, dt \\\ge \int_0^T -(v_0,E_+)\, dt+\inf_{z\in L^2((0,T)\times \Omega;\R^n)}\int_0^T \left[-(z,(I+2B_+)v)+\frac 1 2(z\otimes z, I+2B_+)\right]\, dt\\=\int_0^T -\left[(v_0,E_+)+\frac 12((I+2B_+)v,v)\right]\, dt\\
			=\int_0^T \left[(v\otimes v, I)-(v,\vbs)-\frac 12(v,v)\right]\, dt=\int_0^T \left[\ren(t)-\frac 1 2 (\vbs,\vbs)\right]\, dt. \end{multline*} 
		
		If $G> 0$ a.e. in an open set $\mathcal O$, then \eqref{e:dualsol} obviously yields \eqref{e:dualsolinv}.
	\end{proof}
	
	\begin{remark} Constraints \eqref{e:pd++g} and \eqref{e:constrweakg} may be interpreted as a weak form of a generalized ``transport'' equation \eqref{e:gte1} where the matricial positive-semidefinite ``density'' $G(t,x)$ and the ``velocity field'' $u(t,x)$ may be freely chosen. The simplest pair satisfying \eqref{e:pd++g}, \eqref{e:constrweakg} is $(G,u)\equiv (I,0)$, which yields $\bar v=v$. In this connection, we note that the consistency of the dual formulation of the incompressible Euler equations when the base state coincides with the solution was established in \cite[Theorem 3.6]{ASZ24}.  \end{remark}
	
	\begin{corollary} \label{c:cor} Let $N=1$ (and hence $n=p$), and let $v$ be the gradient of the viscosity solution $\psi$ to the Hamilton-Jacobi equation \eqref{e:mainhj}. Assume that $v_0$ and $v$ are continuous in $\Omega$ and $[0,T]\times \Omega$, resp. Then there exists a sequence $\{\vbs_m\}$ of continuous base states with the following properties: \\

		i) $\vbs_m\to 0$ strongly in $L^1((0,T)\times \Omega;\R^n)$;\\
		
		ii) for every $m\in \mathbb N$, $\mathcal I(v_0, \vbs_m, T)=\mathcal J(v_0, \vbs_m, T)$;\\
		
		iii) for every $m\in \mathbb N$, the dual problem \eqref{e:conc} with $\vbs=\vbs_m$ has a smooth solution $(E_m,B_m)$  satisfying $I+2B_m>0$  in $(0,T)\times \Omega$;\\
		
		iv) the gradient  $v$ of the solution to the Hamilton-Jacobi equation can be fully retrieved from the variational dual solutions  according to the formula \be v=(I+2B_m)^{-1}(\vbs_m-E_m). \ee

	\end{corollary}
	
	\begin{proof} Since $v$ is continuous in $[0,T]\times \Omega$, there exists a sequence $\{u_m\}$ of smooth vector fields on $[0,T]\times \Omega$ that converges to $v$ uniformly. Let $\rho_m:[0,T]\times \Omega\to \R$ be the smooth solution of the linear transport equation \dv{\be \p_t \rho_m+\di (\rho_m u_m)=0 \label{e:trans}\ee} that satisfies the terminal condition \be \rho_m(T,x)=1. \label{e:ttc} \ee Since each $u_m$ is smooth and $\frac d {dt} \rho_m(t)=\rho_m(t)\di u_m$ along any corresponding characteristic, $\rho_m$ is bounded away from $0$ and $+\infty$ (not uniformly in $m$). Set $G_m:=\rho_m I$, and fix a smooth matrix field $\Psi: [0,T]\to X^{n\times n}_s$, $\Psi(0)= 0$. Multiply \eqref{e:trans} by $\tr \Psi$, and integrate by parts over $(0,T)\times \Omega$ to  obtain  \be \label{e:constrweakgrho}\int_0^T \left[((\rho_m-1)I,\p_t \Psi)+(G_m u_m, \nabla \tr \Psi)\right]\, dt=0.\ee In particular, letting $\Psi(t,x)=tI$ gives $$\|\rho_m\|_{L^1((0,T)\times \Omega)}=\|1\|_{L^1((0,T)\times \Omega)}.$$
		Remembering that for $N=1$ we have $L=-\nabla \vv=-\frac 1 2 \nabla \tr$, we conclude that the pairs $(G_m,u_m)$ satisfy the constraints \eqref{e:pd++g} and \eqref{e:constrweakg}  of Theorem \ref{t:smooth}. Let $\vbs_m:=\rho_m(v-u_m)=G_m(v-u_m)$. Then \begin{multline*}\|\vbs_m\|_{L^1((0,T)\times \Omega;\R^n)}\leq \|\rho_m\|_{L^1((0,T)\times \Omega)} \|u_m-v\|_{L^\infty((0,T)\times \Omega;\R^n)}\\ = \|1\|_{L^1((0,T)\times \Omega)} \|u_m-v\|_{L^\infty((0,T)\times \Omega;\R^n)} \to 0.\end{multline*} It remains to apply Theorem \ref{t:smooth} after having observed that  $G_m>0$ everywhere in   $\mathcal O:=(0,T)\times \Omega$ and that the functions $E_m:=-G_m v+\vbs_m=-\rho_m u_m$, $B_m:=\frac 1 2 (G_m-I)=\frac 1 2 (\rho_m-1)I$ are smooth.
	\end{proof}
	\begin{remark} \label{r:vrm} In the pioneering work \cite{CMP18}, Brenier considered the variational dual formulation (with zero base state) for the quadratic Burgers equation --- or, equivalently, to the one-dimensional quadratic H–J --- which corresponds to $n=N=p=1$ and $\bar{v}=0$ in our framework. He shows explicitly that, in that setting, the dual solution $(\rho v,\rho)$ is merely a vector-valued Radon measure (see also Remark \ref{r:rrm} below, which explains the loss of regularity) and that the support of the measure $\rho$ is strictly smaller than $[0,T]\times \Omega$; here $\rho$ can be interpreted as the probability measure transported by $v$ whose marginal at time $T$ is the Lebesgue measure on $\Omega$ (cf. the proof of Corollary \ref{c:cor}). Consequently, $v$ cannot be recovered as the Radon-Nikodym derivative of $\rho v$ w.r.t. $\rho$ outside of the support of $\rho$. Corollary \ref{c:cor} shows that a small perturbation of the base state $\bar{v}=0$ can dramatically improve the regularity of the dual solution and secure one-to-one correspondence between $v$ and the variational dual solution. The same observations apply in the case $N=1$ with $p>1$; under these conditions, the existence of a variational dual solution to H-J with $\bar{v}=0$ --- again a vector-valued Radon measure  --- follows from \cite[Theorem 4.6]{V22}.   \end{remark}
	\begin{remark} The assumption that $v$ is continuous in Corollary \ref{c:cor} is, of course, restrictive, as it rules out the formation of shocks.  Nevertheless, if $v$ is merely essentially bounded, we can take \emph{any} smooth vector field $u_m$ and solve the transport equation \eqref{e:trans} with the terminal condition \eqref{e:ttc}. Then the dual formulation with the (essentially bounded) base state $\vbs_m:=\rho_m(v-u_m)$ is fully consistent (in particular, there is a one-to-one correspondence between the ``primal'' and ``dual'' solutions and no duality gap). Hence, the class of “good’’ base states for $v$ with shocks is still very large.\end{remark}

	We recall the following technical definition, cf. \cite{V22}, that is needed to handle the solvability of the dual problem: \begin{definition} \label{deftr} The operator $L$ is said to satisfy the \emph{trace condition} if the following holds: for any $\zeta\in D(L^*)$ such that  the eigenvalues of the matrix $-L^* \zeta(x)$  are uniformly bounded from above by a constant $k$ for a.e. $x\in \Omega$, the eigenvalues of the matrix $L^* \zeta(x)$ are also uniformly bounded from above a.e. in $\Omega$ by a constant that depends only on $k$, and not on $\zeta$ or on $x$.\end{definition}

	\begin{lemma}[Validity of the trace condition] \label{l:tc} Assume that $N>1$. Then the operator $L$ defined by \eqref{e:opl} satisfies the trace condition.  \end{lemma}
	
	\begin{proof} It suffices to prove that $\vv$ satisfies the trace condition. Using \eqref{e:opu}, for any element $y\in \R^N$ we explicitly compute \begin{gather*}
			(\vv^*(y))_{ijl,iil}=(\vv^*(y))_{iil,ijl}=\frac 12 y_j,\\ (\vv^*(y))_{ijl,hqr}=0\ \mathrm{for}\ \mathrm{all}\ \mathrm{other}\ \mathrm{entries}. 
		\end{gather*} Consequently, \be \label{e:trc} \tr \vv^*(y)=\frac p 2 \sum_{i=1}^N y_i. \ee
		Let $k\ge 0$ be such that $kI+\vv^*(y)\geq 0$. This entails, for every pair of indices $i \neq j$, that $$\left(\begin{array}{@{}c|c@{}}
			2k+ y_i& y_j \\ \hline y_j &
			2k \end{array}\right)=2\left(\begin{array}{@{}c|c@{}}
			k+(\vv^*(y))_{iil,iil} & (\vv^*(y))_{iil,ijl} \\ \hline (\vv^*(y))_{ijl,iil} &
			k+(\vv^*(y))_{ijl,ijl} \end{array}\right)\ge 0.$$ Hence, \be (y_j)^2\leq 4k^2+2k y_i, \quad i \neq j. \ee  It follows by some elementary algebra that the trace \eqref{e:trc} is bounded from above by a constant that merely depends on $k$, and not on $y$. This entails that the eigenvalues of $\vv^*(y)$ are also uniformly controlled from above. 
	\end{proof}

	\begin{theorem}[Existence of a variational dual solution] \label{t:ex} Assume that $N>1$. Then for any $v_0\in X^n$ and $\vbs\in L^2((0,T)\times \Omega;\R^n)$ there exists a maximizer $(E,B)$ to \eqref{e:conc} in the class \eqref{e:be}, and\footnote{\dv{Here, $\iii$ is the averaging operator defined at the beginning of Section \ref{s:se2}, and the functional $\mathfrak C$ is defined, mutatis mutandis, as in \eqref{cccd}.} } $\cci \leq \mathcal J(v_0, \vbs, T)< +\infty$.  \end{theorem}

	\begin{proof}  It suffices to consider the pairs $(E,B)$ that meet the restrictions \eqref{e:constrweak}, \eqref{e:bwe}. In particular, the pair $(E,B)=(\iii \vbs,0)$ satisfies \eqref{e:constrweak}, \eqref{e:bwe} and hence is admissible. (Indeed, for this pair $\int_0^T \left[(B,\p_t \Psi)+(E, L \Psi)\right]\, dt=\int_0^T (\di \iii \vbs, \vv \Psi)\, dt=0$ for all smooth  $\Psi: [0,T]\to X^{n\times n}_s$.) Testing \eqref{e:conc} with this pair, and remembering that $|\Omega|=1$, we see that \begin{multline*}\mathcal J(v_0, \vbs, T)\geq \int_0^T \left[-(v_0,\iii \vbs)+\frac 1 2 (\vbs,\vbs)\right]\, dt+\mathcal K(\iii\vbs- \vbs,0)\\ = \int_0^T \left[-(\iii v_0,\iii \vbs)+\frac 1 2 (\vbs,\vbs)\right]\, dt-\frac 1 2 \int_0^T (\vbs-\iii \vbs,\vbs- \iii \vbs)\, dt\\=\int_0^T \left[-(\iii v_0,\iii \vbs)+\frac 1 2 (\iii \vbs,\iii \vbs)\right]\, dt=\cci.\end{multline*} Let $(E_m,B_m)$ be a maximizing sequence.  Since $\cci \le \mathcal J(v_0, \vbs, T)$, without loss of generality we may assume that \be\label{e:ms}\cci \le \int_0^T \left[-(v_0,E_m)+\frac 1 2 (\vbs,\vbs)\right]\, dt+\bigkm.\ee The eigenvalues of $-B_m$ are uniformly bounded from above because $I+2B_m\geq 0$. Consequently, Lemma \ref{l:tc} and Definition \ref{deftr} yield  $I+2B_m\leq kI$ with some constant $k>0$ a.e. in $(0,T)\times\Omega$. By the definition of $\mathcal K$ in \eqref{e:defk1}, we have \be \label{e:ms1} \bigkm\leq  \inf_{z\otimes z\le M}\int_0^T\left[(z,E_m-\vbs)+\frac k 2(M, I)\right]\,dt=-\frac 1{2k}\int_0^T (E_m-\vbs,E_m-\vbs)\,dt.\ee We infer that 
		\begin{multline} \label{e:ms2} \frac 1{4k} \int_0^T(E_m,E_m) \, dt \leq \frac 1{2k} \int_0^T(E_m-\vbs,E_m-\vbs)\, dt +\frac 1{2k} \int_0^T(\vbs,\vbs)\, dt\\ \leq -\int_0^T (v_0,E_m)\, dt+\frac {k+1}{2k} \int_0^T (\vbs,\vbs) \, dt-\cci\\ \leq \frac {k+1}{2k} \int_0^T[(\vbs,\vbs)+2k(v_0,v_0)]\, dt-\cci+\frac 1{4(k+1)} \int_0^T(E_m,E_m),\end{multline} which gives a uniform $L^2((0,T)\times \Omega;\R^n)$-bound on $E_m$. Taking into account \eqref{e:ms1} and \eqref{e:ms2}, we infer that the right-hand side of \eqref{e:ms} is uniformly bounded, whence $\mathcal J(v_0, \vbs, T)<+\infty$. The functional $\mathcal K$ is concave and upper semicontinuous on $L^2((0,T)\times \Omega;\R^n)\times L^\infty((0,T)\times \Omega;\R^{n\times n}_s)$ as an infimum of affine continuous functionals, cf.  \eqref{e:defk1}. The functional $\int_0^T (v_0,\cdot)\, dt$ is a linear bounded functional on $L^2((0,T)\times \Omega;\R^n)$. Consequently, every weak-$*$ accumulation point of $(E_m,B_m)$ is a maximizer of \eqref{e:conc}. Note that the constraints \eqref{e:constrweak}, \eqref{e:bwe} are preserved by the limit. 
	\end{proof}
	
	\begin{remark} The trace condition is not valid for $N=1$, cf. \cite{V25}, which explains why, as already mentioned in Remark \ref{r:vrm}, for $N=1$ and $\bar{v}=0$ the variational dual solution to the quadratic H–J is merely a vector Radon measure and does not belong to the class \eqref{e:be}. \label{r:rrm}\end{remark}

	\section{The utility of base states in practical considerations: A formal convex gradient flow scheme}\label{sec:dual_grad_flow_setting}

	In this section we describe a formal scheme for obtaining weak solutions to a class of PDEs; the class includes many equations of continuum mechanics. The scheme is quite generally applicable, up to minor adjustments, including the case of the Nash system, see Section \ref{sec5} below. Comparable adaptive numerical schemes employing variational dual solutions and adjustable base states --- but not based on the gradient flow idea developed here --- for solving nonlinear differential equations have recently been demonstrated in  \cite{ka1,sga,ka2,kpa}. These methods were developed in connection with the nonlinear system of ODEs of Euler for the dynamics of a rigid body, nonconvex elastostatic and dynamics of a bar (without regularization), (inviscid) Burgers equation, and the problem of traveling waves of a dispersive, nonlocal, nonlinear semi-discrete Burgers equation.
	
	\subsection{Variational setup} \label{sec:var_set}
	Following \cite{ach3} for the notation and setup of the problem, let lower-case Latin indices belong to the set $\{1,\cdots, m\}$ representing rectangular Cartesian spatial coordinates, $t$ is time, and we will employ the Einstein summation convention. Let upper-case Latin indices belong to the set $\{1,2,3, \cdots,u \}$, indexing the components of the array of primal variables, $U \in \R^u$. This array contains variables whose partial derivatives, w.r.t $(t,x)$, of at most first order appear in the governing system of primal equations ---- thus, a conversion to a first-order system of the governing PDE is employed. Consider the system of equations
	\begin{subequations}
		\label{eq:gov_cont_mech}
		\begin{align}
			& \calC_{\Gamma I} \p_t U_I + \p_{x_j} (\calF_{\Gamma j}\ach{(U)}) + G_\Gamma\ach{(U)} + V_\Gamma(t,x)  = 0 \ \mbox{in} \ \Omega \times (0,T),  \qquad \Gamma = 1, \ldots, d, \label{eq:gov}\\
			& \calC_{\Gamma I} U_I (x, 0)  = \calC_{\Gamma I} U_I^{(0)} (x) \ \mbox{specified on } \Omega \ \mbox{(initial conditions)},\\
			& (\calF_{\Gamma j}\ach{(U)}n_j)(t,x)  = (B_{\Gamma j} n_j)(t,x)  \ \mbox{specified on } \p\Omega_\Gamma \ \mbox{(boundary conditions)} \label{eq:gov_bc},
		\end{align}
	\end{subequations}
	Note that, \dv{throughout Section} \ref{sec:dual_grad_flow_setting},  $\Omega$ denotes a given bounded domain in $\R^m$  with sufficiently regular boundary $\p \Omega \supset  \bigcup_\Gamma \p \Omega_\Gamma$ and unit normal vector $n$, and the upper-case Greek indices index the number of equations. Here, $\calC \in \R^{d \times m}$ is a fixed matrix, and $\calF, G, V, B$ and $U^{(0)}$ are, respectively, $\R^{d \times m}$-, $d$-,$d$-, $\R^{d \times m}$- and $u$-dimensional given functions of their argument.

	Define the pre-dual functional by forming the scalar products of \eqref{eq:gov} with the dual fields $D$, $(t,x) \mapsto D(t,x) \in \R^d$ (see \eqref{eq:gov}), integrating by parts, substituting the prescribed initial and boundary conditions (ignoring, for now, space-time boundary contributions that are not specified) \textit{and subtracting an auxiliary potential $H$ as shown}:
	\begin{multline}
		\label{eq:predual}
		\widehat{S}_H[U,D] \\ = \int_\Omega \int_0^T \left( - \calC_{\Gamma I} U_I \p_t D_\Gamma - \calF_{\Gamma j}(U) \p_{x_j} D_\Gamma + \left( G_\Gamma(U) + V_\Gamma(t,x) \right)D_\Gamma - H(U,t,x) \right) \,dtdx  \\
		\qquad - \int_\Omega \calC_{\Gamma I} {U}_I^{(0)}(x) D_\Gamma(x, 0) \, dx + \sum_{\Gamma \ach{=1}}^{\ach{d}} \int_{\p \Omega_\Gamma} \int_0^T  B_{\Gamma j} \, D_\Gamma \, n_j \, dt\,\da,   
	\end{multline}
	where the arguments $(t,x)$ are suppressed except to display the  explicit dependence of $V, H$ and in the initial condition. \dv{Although we adopt the summation convention, we write the last term more explicitly because it involves more than two repeated indices $\Gamma$.}

	Define
	\begin{equation}
		\label{eq:defs}
		\begin{aligned}
			(t,x) \mapsto \dee(t,x) & := (\p_t D(t,x), \nabla D(t,x), D(t,x)) \in \R^l, \quad l := d + md + d,\\
			\ach{ \scl_H(U,\dee,t,x)} & := \ach{- \calC_{\Gamma I} U_I \p_t D_\Gamma - \calF_{\Gamma j}(U) \p_{x_j} D_\Gamma + G_\Gamma(U) D_\Gamma - H(U,t,x).}
		\end{aligned}
	\end{equation}
	The $(t,x)$-dependence of $H$, introduced in \cite[Sec.~5]{ach2} with the stated goal of aiding in the recovery of non-unique solutions of the primal problem, includes the deployment of \emph{base states}, cf. Section \ref{sbs}, in the design of the dual functional\footnote{As an example, in \cite[Sec.~5.4]{ka2} piecewise-in-time solutions of Burgers with small viscosity are used as base states to recover entropy solutions to inviscid Burgers, when written in H-J form.}. These are prescribed fields $(t,x) \mapsto \bar{U}(t,x)$. \dv{From now on, $H$ can depend on $(t,x)$ only through $\bar{U}$. Hence, in what follows, we slightly abuse the notation and write $H(U,\bar U)$, $\scl_H(U,\dee,\bar U)$, etc.} We now require that the choice of $H$, given $t,x$ and a base state $\bar{U}$, be such that there exists a non-empty, open, convex subset   $\sco=\sco(t,x)\subset \R^l$ and a function $U^{(H)}: \sco\times \mathbb R^{u}\to \mathbb R^{u}$ satisfying
	\begin{subequations}\label{eq:dtp_rel}
		\begin{align}
			& \p_U \scl_H \Big( U^{(H)}(\dee, \bar{U}), \dee,\dv{\bar U} \Big)  = 0 \quad \forall \quad \dee \in \sco,  \label{eq:dtp} \\
			& \mbox{and} \notag\\
			&  \p_U \scl_H \Big( U^*, 0,\dv{\bar U}  \Big)  = 0 \mbox{ admits the unique solution } U^* = \bar{U}(t,x). \label{eq:U*_dee0}
		\end{align}   
	\end{subequations}
	We refer to such a function $U^{(H)}$ as a \emph{Dual-to-Primal} (DtP) mapping. Assume that a \textit{dual-to-primal} (DtP) ``change of variables'' mapping $U^{(H)}$ exists. Given $D$ that \begin{enumerate}
		\item satisfies $\dee(t,x) \subset \sco(t,x)$, 
		\item has a well-defined trace on the space-time boundary,
		\item satisfies the boundary constraints on the parts of the space-time  boundary \emph{complementary} to those that appear explicitly in \eqref{eq:gov_cont_mech}: \be \label{e:cbc} D_\Gamma=D^*_\Gamma\quad \mathrm{for}\quad t=T\quad\mathrm{or}\quad x\not\in \p \Omega_\Gamma,   \qquad \Gamma = 1, \ldots, d,\ee where $D^*: [0,T] \times \overline\Omega\to \R^d$ is a fixed, sufficiently regular, \emph{arbitrarily chosen} function satisfying $\dee^*(t,x) \subset \sco$, 
	\end{enumerate} we define the \textit{dual} functional as
	\begin{multline}\label{eq:dual}
		S_H[D; \bar{U}] := \widehat{S}_H \left[ U^{(H)}(\dee,\bar{U}), D \right] 
		\ = \int_\Omega \int_0^T  \scl_H\left(U^{(H)}(\dee,\bar{U}), \dee,\bar{U} \right) \, dtdx\\ \dv{+\int_\Omega \int_0^T  V_\Gamma(t,x) D_\Gamma \, dtdx} - \int_\Omega \calC_{\Gamma I} U_I^{(0)}(x) D_\Gamma(x, 0) \, dx \\
		\qquad + \sum_{\Gamma \ach{= 1}}^{\ach{d}} \int_{\p \Omega_\Gamma} \int_0^T  B_{\Gamma j} \, D_\Gamma \, n_j \, dt\, \da.
	\end{multline} 
	
	Using the facts
	\begin{enumerate}
		\item \eqref{eq:dtp},
		\item $\scl_H$ in \eqref{eq:defs} is necessarily affine in its argument \ach{$\dee$ by construction (related to the idea of ``Lagrange multipliers'' applied to ``constraints'' in forming the pre-dual functional $\widehat{S}_H$)}, 
		\item the variations $\delta D$ vanish on the parts of the boundary where Dirichlet boundary conditions on the dual fields are specified, cf. \eqref{e:cbc},
	\end{enumerate} 
	it is easily verified that, formally, the Euler-Lagrange equations and side conditions of $S_H$ \eqref{eq:dual}  are given by the system \eqref{eq:gov_cont_mech}, using the substitution $U \to U^{(H)}(\dee,\bar{U})$, for \textit{any} $H$ that allows for the construction of $U^{(H)}$.
	
	In the following, we make the choice $D^* = 0$ for simplicity.
	\begin{remark}
		The flexibility afforded by the arbitrary choice of $D^*$ can be practically beneficial, especially when solving a problem with a fixed base state (when using a changing base state scheme, the choice $D^* = 0$ is most natural as discussed below). Computational experience with even linear problems like the heat and the linear transport equation in \cite{ka1,sa} with $\bar{U} = 0$ and $D^* = 0$ shows the development of strong boundary layers and/or large gradients of dual approximations at domain corners. These arise from the demands on the dual solution placed by meeting the arbitrarily set dual Dirichlet boundary condition and the needs of satisfying the primal problem. An obvious solution for such a situation is to note the values of the dual fields in the interior of the domain corresponding to an approximate solution obtained with  a first prescription of say $D^* = 0$, smoothly extrapolate, with small gradients, that `interior' field to the boundaries, and then apply the values obtained on the `Dirichlet boundaries' as the modified $D^*$ boundary condition to obtain the exact solution or better approximation.
	\end{remark}
	Before proceeding further, we first show that choices of $H(U,\bar U)$ satisfying \eqref{eq:dtp_rel} are easily made\footnote{From now on, we employ an abuse of notation and write $H(U,\bar U)$, $\scl_H(U, \dee , \bar{U})$ instead of $H(U,t,x)$, $\scl_H(U, \dee, t,x)$, etc.}; for instance, suppose 
	\begin{equation}\label{eq:convexity_H}
		H \mbox{ is a strictly convex function of } U \mbox{ attaining its unique minimum at } U = \bar{U},
	\end{equation}
	e.g.,
	\begin{equation*}
		H (U, \bar{U}) = a \, \half |U - \bar{U}|^2 + b \, \frac{1}{p} |U - \bar{U}|^p, \qquad a, b, p \in \R; \quad a > 0; \quad b \geq 0; \quad p > 2.
	\end{equation*}
	Then, since $\scl_H$ is necessarily affine in $\dee$,
	\[
	\scl_H(U, 0, \bar{U}) = - H(U, \bar{U})
	\]
	and
	\[
	\p_U \scl_H(U^*,0, \bar{U}) = -\p_U H(U^*,\bar{U}) = 0
	\]
	admits the unique solution $U^* = \bar{U}$ and
	\begin{equation}\label{eq:L_+def}
		- \, \p_{UU} \scl_H (\bar{U}, 0, \bar{U}) >0
	\end{equation}
	in the sense of symmetric, positive-definite matrices. 
	
	By the implicit function theorem\footnote{Remember that here we work at a formal level and ignore regularity issues.}, there exists a (w.l.o.g. convex) neighborhood $\sco=\sco(t,x) \subset \R^l$ containing $\dee = 0$ and a mapping $U^{(H)}(\cdot \, , \bar{U}) : \sco \to \R^u$ satisfying \eqref{eq:dtp}, for each fixed choice of $t,x$ and $\bar{U} \in \R^u$.
	
	We also make the following observation, cf. \cite{CMP18}. Let us define
	\begin{equation*}
		\begin{aligned}
			& \widetilde{S}_H [D; \bar{U}]  \ \ach{:=} \ \sup_U \widehat{S}_H[U,D] \\
			& = \sup_U \int_\Omega \int_0^T \scl_H\left(U(t,x), \dee(t,x), \bar{U}(t,x) \right) \, dtdx \\
			& \qquad \ach{
				+ \ (\mbox{space-time bulk and  boundary terms from \eqref{eq:dual}})}\\
			& = \int_\Omega \int_0^T \sup_U \scl_H\left(U, \dee(t,x), \bar{U}(t,x) \right) \, dtdx \\
			& \qquad \ach{
				+ \ (\mbox{space-time bulk and boundary terms from \eqref{eq:dual}})}\\
		\end{aligned} 
	\end{equation*}
	(noting that $\scl_H$ depends on the field $U$ only through its pointwise values). Since $\widehat{S}$ is affine\footnote{Rigorously speaking, $\widehat{S}_H[U,\cdot]$ has a certain convex domain in $\mathcal H$ where $\dee$ and the boundary traces are well-defined, and the Dirichlet boundary condition \eqref{e:cbc} holds; outside this domain, we simply set $\widehat{S}_H[U,D]:=+\infty$; this is enough to guarantee the convexity of $\widetilde{S}_H$.} in $D$, the functional $\widetilde{S}_H[ \ \cdot \ ; \bar{U}]$ is  \emph{convex} in the Hilbert space $$\mathcal H:=L^2(\Omega\times(0,T);\R^d).$$  Since the set of positive-definite matrices is open in $\R^{u \times u}$, using \eqref{eq:L_+def} we infer that for a.e. $(t,x) \in [0,T] \times \Omega$ there exists a $\R^l \times \R^u$-neighborhood centered in $(0,\bar{U}(t,x))$ in which $\scl_H\big(\cdot, \dee(t,x), \bar{U}(t,x)\big)$ is strictly concave in $U$ w.r.t. its first argument. Consequently, for a given base state $\bar{U}: [0,T] \times \Omega \to \R^u$, there exists a convex set $\sco^*_{\bar{U}} \subset \mathcal H$ containing zero such that for every $D \in \sco^*_{\bar{U}}$ one has
	\begin{equation}\label{eq:Dtp_supremizer}
		\begin{aligned}
			\sup_U \scl_H(U,\dee(t,x), \bar{U}(t,x)) & = \scl_H\Big(U^{(H)}(\dee(t,x), \bar{U}(t,x)), \dee(t,x), \bar{U}(t,x) \Big), \\
			& \qquad \qquad  \mathrm{for}\ \mathrm{a.e.} \quad (t,x) \in [0,T] \times \Omega 
		\end{aligned}
	\end{equation}
	and
	\begin{equation}\label{eq:SH=tildeSH}
		S_H[D;\bar{U}] = \widetilde{S}_H [D; \bar{U}].
	\end{equation}
	
	In other words, for each fixed base state $\bar{U}: [0,T] \times \Omega \to \R^u$ there exists a convex set $\sco^*_{\bar{U}}\subset \mathcal H$ in which the dual functional $S_H[D; \bar{U}]$ is convex and its Euler-Lagrange equations and side conditions are given by \eqref{eq:gov_cont_mech}, with the substitution $ U(t,x) \to U^{(H)}(\dee(t,x), \bar{U}(t,x))$. The DtP map $(\dee,\bar{U}) \mapsto U^{(H)}(\dee,\bar{U})$ is well-defined for $D\in \sco^*_{\bar{U}}$, with the maximizer in \eqref{eq:SH=tildeSH} being an isolated critical point of $\scl_H(\,\cdot \,, \dee(t,x); \bar{U}(t,x))$. 
	
	It is this insight that we use next to propose a convex gradient flow technique to attempt to find weak solutions to \eqref{eq:gov_cont_mech}, utilizing minimizers of \eqref{eq:SH=tildeSH} for a suitable $H$ \textbf{designed by the scheme}.
	\begin{remark} In \cite{sga} the minimizers of  $\tilde{S}_H$ were called \textit{variational dual solutions} and the ones among these that allow the recovery of a weak solution of the original PDE via the DtP map were defined as \textit{dual solutions}. A meaningful choice of the auxiliary potential $H$ is essential for the idea of variational dual solutions  (as defined in \cite{sga}) to be non-vacuous, and they can differ from dual solutions, as happens, in particular, in the presence of a duality gap. \end{remark}
	Before proposing the gradient flow scheme for our problem, we make a few further observations:
	\begin{enumerate}    
		\item The ``$L^2$-variation'' $\frac{\delta S_H}{\delta D}\big[D;\bar{U}\big]$ is \textbf{formally} (by the envelope theorem) defined by \begin{equation}
			\label{eq:predual2}
			\begin{aligned}
				& \int_\Omega \int_0^T \frac{\delta S_H}{\delta D}\big[D(t,x); \bar{U}(t,x) \big] \, \delta D(t,x)\, dtdx \\  &  = \int_\Omega \int_0^T \left( - \calC_{\Gamma I} U_I \p_t \delta D_\Gamma - \calF_{\Gamma j}(U) \p_{x_j} \delta D_\Gamma + \left( G_\Gamma(U) + V_\Gamma(t,x) \right)\delta D_\Gamma \right) \,dtdx  \\
				& \qquad - \int_\Omega \calC_{\Gamma I} {U}_I^{(0)}(x) \delta D_\Gamma(x, 0) \, dx + \sum_{\Gamma \ach{=1}}^{\ach{d}} \int_{\p \Omega_\Gamma} \int_0^T  B_{\Gamma j} \, \delta D_\Gamma \, n_j \, dt\,\da,
			\end{aligned}    
		\end{equation} with the substitution $U(t,x) = U^{(H)}(\dee(t,x), \bar{U}(t,x))$.

		\item We note that 
		\begin{equation} \frac{\delta S_H}{\delta D}\big[D;\bar{U}\big]\quad \textrm{depends on (the function)}\ D\ \textrm{only through}\ U^{(H)}(\dee, \bar{U}). \label{eq:var_deriv_dep_Uh} \end{equation}
		
		\item Solving system \eqref{eq:gov_cont_mech} in a weak sense means solving
		\begin{equation*}
			\int_\Omega \int_0^T \frac{\delta S_H}{\delta D}\big[D(t,x); \bar{U}(t,x) \big] \, \delta D(t,x)\, dtdx = 0
		\end{equation*}
		for $D$ satisfying the Dirichlet BC \eqref{e:cbc} and for all $\delta D$ satisfying homogeneous Dirichlet BC at the complementary boundary defined in \eqref{e:cbc}. 
		\item Suppose $\bar{U}$ is a weak solution to \eqref{eq:gov_cont_mech}. Then $D = 0$ is a solution to the dual problem with $D^*=0$, and defines, through the DtP map, a global in time solution to the primal problem. Thus, the intuition is that if $\bar{U}$ is ``close'' to a primal solution then, given the convexity properties of the dual problem, it is reasonable to expect to obtain a primal solution (through the DtP map), to which the base state is close, by solving the dual problem. 

	\end{enumerate}
	\subsection{The scheme: Gradient flow evolutions in stages.}\label{sec:grad_flow_iter}
	\begin{enumerate}
		\item We think of a (fake, time-like) variable $\s$ along which a gradient descent in the Hilbert space $\mathcal H$ of the functional $S_{H_k}$ is executed in stages indexed by $k$. Here, $H_k$ is the auxiliary potential used in the stage $k$. A typical choice is \begin{equation*}
			H_k(U) = a \, \half |U - \bar{U}^k|^2 + b \, \frac{1}{p} |U - \bar{U}^k|^p, \qquad a, b, p \in \R; \quad a > 0; \quad b \geq 0; \quad p > 2, 
		\end{equation*} where the sequence of the base states $\bar{U}^k(t,x)$ will be described below. 
		\item Each stage comprises the interval $[0,\tilde{\s}_k^*)$, where $\tilde{\s}_k^*$ could possibly be $+\infty$, but will be otherwise determined by the scheme as follows.
		\item We consider dual unknown functions $D^k:[0,T] \times \Omega \times {[0,\tilde{\s}_k^*)} \to \R^d, (t,x,\s) \mapsto D^k(t,x,\s)$, and corresponding variations (test functions) $\delta D$ to execute the following gradient flow:
		
		for all $\s \in [0,\tilde{\s}_k^*]$
		find $D^k$ satisfying the dual boundary conditions by solving,
		for all $\delta D$ satisfying their respective homogeneous boundary conditions,
		\begin{multline}\label{eq:grad_flow}
			\int_{[0,T] \times \Omega}
			\delta D(t,x)\,
			\frac{\partial D^k}{\partial \s}(t,x,\s)\, dt\,dx
			\\=
			- \int_{[0,T] \times \Omega}
			\frac{\delta S_{H_k}}{\delta D}[D^k;\bar{U}^k](t,x,\s)\,
			\delta D(t,x)\, dt\,dx .
		\end{multline}
		\item Remember that the dissipation of the convex driving energies is non-increasing along such gradient flows (see, e.g., \cite{Atbu14}): 
		\begin{equation}\label{eq:norm_grad_flow} \frac d{d\s}\int_{[0,T] \times \Omega}
			\Big|\frac{\delta S_{H_k}}{\delta D}\Big|^2[D^k;\bar{U}^k](t,x,\s)\,
			\, dt\,dx \le 0.\end{equation}
		\item We also denote 
		\[
		U^{H_k}(t,x,\s) := U^{(H_k)}\big( \dee^k(t,x,\s), \bar{U}^k(t,x) \big)
		\] provided the right-hand side is well-defined.

		\item For $k = 1$, we select the base state $\bar{U}^1(t,x)$ arbitrarily, noting that the closer such a base state is to a solution, the better the chances of the scheme to converge to a solution. 
	\item We now run the first gradient descent \eqref{eq:grad_flow} emanating from $D^1(t,x,0) := 0$. Here and below, recall that the dual boundary data is defined by $D^* = 0$.
	
	\item Then, either the gradient flow equilibrates, possibly at $\tilde{\s}^*_1 =+ \infty$, without violating 
	\begin{equation}\label{eq:base_state_change_crit} \begin{aligned}
			& \dee^1(t,x,\s) \in \sco(t,x),\\
			- \p_{UU} \scl_H&\Big(U^{(H)}(\dee^1(t,x,\s), \bar{U}^1(t,x,\s)), \dee^1(t,x,\s), \bar{U}^1(t,x) \Big) > 0, \end{aligned} \end{equation}
	for a.e.  $(t,x) \in [0,T] \times \Omega$ and for all  $\s \in [0,\tilde{\s}^*_1]$, 
	so that, as a consequence,
	\begin{multline*}\sup_U \scl_H(U,\dee^1(t,x,\s), \bar{U}^1(t,x)) 
		\\ = \scl_H\Big(U^{(H)}(\dee^1(t,x,\s), \bar{U}^1(t,x,\s)), \dee^1(t,x,\s), \bar{U}^1(t,x) \Big),\end{multline*}
	or the flow reaches a fake time $\tilde{\s}^*_1$ when \eqref{eq:base_state_change_crit} is violated on a set of positive Lebesgue measure in $[0,T]\times \Omega$.
	
	\item If the first alternative holds, then we can apply the DtP map to $D^1(t,x,\tilde{\s}^*_1)$ and generate a solution to \eqref{eq:gov_cont_mech} by \be U(t,x):=U^{H_1}(t,x,\tilde{\s}^*_1)=
	U^{(H_1)}\!\left(
	\dee^1(t,x,\tilde{\s}^*_1),
	\bar{U}^1(t,x)
	\right).\ee \begin{remark} In cases where the solutions to \eqref{eq:gov_cont_mech} are not unique, we can find only one of them. However, by varying the initial base state $\bar{U}^1$, it is expected to be possible to capture additional solutions as well. In principle, any solution can be retrieved, using observation (3) at the end of Sec.~\ref{sec:var_set}. In practice, see the approximation of a) a wide variety of traveling waves of a dispersive semi-discrete Burgers equation in \cite{kpa} (with an algorithm which does not exploit the feature \eqref{eq:norm_grad_flow}), and b) even unstable solutions of the primal problem in a stable manner in \cite[Sec.~6.3.4]{sga} that deals with the nonconvex elastodynamics of a bar (utilizing only one base state and stage).\end{remark}

	\item If the second alternative holds and $\tilde{\s}^*_1$ is finite, we define $\s_1^* := \tilde{\s}^*_1 - \nu$, for some user-defined tolerance $\nu > 0$. Moreover, if $\tilde{\s}^*_1 =+ \infty$ --- which means that the limiting $D(\tilde{\s}^*_1)$ lies outside of the ``DtP zone'' \eqref{eq:base_state_change_crit} although the trajectory $D(\s)$ stays inside the ``DtP zone'' --- we define $\s_1^* := \mu$, for some user-defined big number $\mu > 0$. 
	\item  Then we set
	\begin{subequations}\label{eq:base_state_change}
		\allowdisplaybreaks
		\begin{align}
			& \bar{U}^2(t,x) :=U^{H_1}(t,x,\s_1^*)= U^{(H_1)}\big( \dee^1(t,x,\s_1^*), \bar{U}^1(t,x) \big) \qquad \mbox{ for all } \s \in [0, \tilde{\s}^*_2], \tag{\ref{eq:base_state_change}}\\
			& D^2(t,x,0) := 0. \notag
		\end{align}
	\end{subequations}
	
	\item We now execute the second gradient flow to find $D^2$. By \eqref{eq:U*_dee0} and \eqref{eq:base_state_change} we always have $$U^{H_2}(t,x,0) = \bar{U}^2(t,x)=U^{H_1}(t,x,\s_1^*),$$ i.e., the new $U$-trajectory is a continuation of the old one. 
	
	\item We continue the iterations until equilibration is achieved at some stage $k_*\in \mathbb N$, allowing us to generate a solution $$U(t,x):=U^{H_{k_*}}(t,x,\tilde{\s}^*_{k_*}).$$
	\item Since
	\[
	U^{H_{k+1}}(t,x,0) = \bar{U}^{k+1}(t,x) = U^{H_k}(t,x,\s^*_k),
	\]
	it follows from \eqref{eq:var_deriv_dep_Uh} and \eqref{eq:norm_grad_flow} that the $\mathcal H$-norm of the gradients $\| \frac{\delta S_{H_k}}{\delta D} \|_{\mathcal H}$, i.e., the dissipation of the energies does not increase with the course of fake-time, \textbf{including the moments of switching between the stages} (with the goal of eventually approaching zero dissipation via the gradient flow).
	\item It remains to prove that one has equilibration after a finite number of steps, i.e., $k_*\in \mathbb N$. For analytical purposes $k_* \to +\infty$ is acceptable, and for computational purposes reaching the threshold $\| \frac{\delta S_{H_k}}{\delta D} \|_{\mathcal H} \leq  \tau$, for some user-defined tolerance ($0 < \tau \ll 1$), in a finite number of stages is sufficient.
	
	\end{enumerate}  
	
	\begin{remark} It is natural to expect that a starting base state in time-dependent problems can be more effective if the real-time interval $[0,T]$ is divided into smaller disjoint sub-intervals whose closed union covers $[0,T]$, with the dual problem solved by the above gradient flow (in fake time) in each such sub-interval, defined with the real-time initial condition generated by the solution of the previous sub-interval. Thus, in this strategy, base states are switched both in real-time (to provide constant-in-real-time starting base state guesses for each sub-interval, at a minimum), as well as potentially in fake-time for the gradient flow within each sub-interval of real time. In experience with nonlinear transient problems computed without the gradient flow strategy \cite[Sec.~5.3]{ka1}-\cite{ka2}, such switching of base states across sub-intervals in real-time was found to be essential for success (as well as computational efficiency). On the other hand, \cite{kpa} involves a time-independent problem where switching base states between stages of Newton method based iterations (the analog of the gradient flow stages) was found to be essential.
	\end{remark}
	
	\section{A toy example of the gradient flow performance} \label{s:toy}
	
	In this section, we illustrate (using a very simple model) how the switching of the base states, the convergence of the scheme, and the eventual generation of a solution occur in practice.
	
	Consider the following \emph{algebraic} system of equations (so that there is no dependence on $t$ and $x$ and no integration) for the unknown vector $U=(U_1,U_2)\in \mathbb R^2$:  
	\be \label{e:alg} \mathfrak Q_c(U)=0, \ee where \be \mathfrak Q_c(U)_i:=(U_1-c)(U_2-c)-(U_i-c), \ i=1,2.
	\ee Here $c$ is a scalar parameter (in what follows, we tacitly assume $c\neq - \frac 1 2$; the special case $c= - \frac 1 2$ is addressed in Remark \ref{r:c12}). It is obvious that \eqref{e:alg} has two solutions: \[(U_1,U_2)=(c,c), \quad (U_1,U_2)=(c+1,c+1).\]
	Let us see how our gradient flow scheme finds one of them.
	
	For any given base state $\bar{U}\in \mathbb R^2$, define the auxiliary potential \be \label{auxp}H(U)=H(U,\bar U):=\frac 1 2 |U-\bar{U}|^2.\ee
	Mimicking \eqref{eq:predual}, define the pre-dual functional by forming the scalar product of \eqref{e:alg} with the dual variable $D=(D_1,D_2)\in \mathbb R^2$:
	\begin{equation}
	\label{eq:predualalg}
	\scl_H(U,D,\bar U):=\widehat{S}_H[U,D] := (U_1-c)(U_2-c)(D_1+D_2)-(U_1-c)D_1-(U_2-c)D_2-H(U).
	\end{equation}
	Taking the variation w.r.t. $U$, we observe that the DtP map $U^{(H)}$ is defined (provided $|D_1+D_2|\neq 1$) by the formula $U^{(H)}(D)=u$, where $u=(u_1,u_2)$ is the unique solution of the linear system \be \label{e:linsys}(u_2-c)(D_1+D_2)-D_1-(u_1-\bar{U}_1)=0, \quad (u_1-c)(D_1+D_2)-D_2-(u_2-\bar{U}_2)=0.\ee
	Moreover,
	\be \label{e:hesl}
	- \, \p_{UU} \scl_H (U, D, \bar{U}) =\left(\begin{array}{@{}c|c@{}}
	1& -D_1-D_2 \\ \hline -D_1-D_2 &
	1 \end{array}\right).
	\ee The ``DtP zone'' where the evolution of the gradient flow scheme should occur is determined by the following condition, cf. \eqref{eq:base_state_change_crit}, \be \label{eq:base_state_change_crit_alg}
	- \p_{UU} \scl_H (U^{(H)}(D), D, \bar{U}) > 0, 
	\ee
	which in view of \eqref{e:hesl} is simply equivalent to \be \label{e:good} |D_1+D_2|< 1.\ee
	For such $D$, we define \be S_H[D] := \widehat{S}_H \left[ U^{(H)}\dv{(D)}, D \right]. \ee This is actually a restriction  of the convex functional \[\widetilde{S}_H [D]  := \sup_U \widehat{S}_H[U,D], \quad D\in \mathbb R^2,\] to the ``DtP zone'' \eqref{e:good}. \dv{(Indeed,  for $D$ satisfying \eqref{e:good}, equality \eqref{e:hesl} implies that $\widehat{S}_H$ is strictly concave with respect to $U$. Hence, $\widehat{S}_H(\cdot,D)$ necessarily attains its maximum at its critical point $U^{(H)}(D)$.)} A tedious calculation involving \eqref{e:linsys} shows that \[ S_H[D]=\frac 1 2 \left(D-\bar U+\left(\begin{array}{@{}c@{}}
	c \\ \hline c \end{array}\right)\right)^\top\left(\begin{array}{@{}c|c@{}}
	1& -D_1-D_2 \\ \hline -D_1-D_2 &
	1 \end{array}\right)^{-1} \left(D-\bar U+\left(\begin{array}{@{}c@{}}
	c \\ \hline c \end{array}\right)\right)-\frac{|c-\bar{U}|^2} 2,  \]
	whence $S_H$ is strictly convex within its domain of definition. Moreover, it follows from the envelope theorem that \be \label{envelo} \nabla S_H(D)=\mathfrak Q_c(U^{(H)}(D)).\ee

	For definiteness (and somewhat echoing \cite{CMP18}), let our gradient flow scheme start from the zero base state \[
	\bar{U}^1 =(0,0) \quad \mbox{for all } \s \in [0,\tilde{\s}_1^*],\] 
	where $\tilde{\s}_1^*$ is the duration of the first stage (to be specified below). We run the gradient flow 
	\be \frac d {d\s} D^1(\s)=-\nabla S_{H_1}[D^1(\s)], \quad D^1(0)=0 \label{e:ourgf}\ee ---  in our toy model the ambient Hilbert space is two-dimensional and the trajectory\footnote{Remember that the upper index indicates the number of the stage of the scheme and the lower indices denote the components of the trajectory.} 
	$(D^1_1,D^1_2)$ of the gradient flow  is determined by a pair of ODEs  --- starting from $(D^1_1,D^1_2)(0)=(0,0)$  until it either reaches the boundary of the ``DtP zone'' in (fake) time $\tilde{\s}_1^*$ or  equilibrates, possibly at $\tilde{\s}_1^*=+\infty$. Since $\bar{U}_1^1=\bar{U}_2^1$, by symmetry and strict convexity of $S_{H_1}$ the straight line \be D_1=D_2 \label{e:ans} \ee is invariant for \eqref{e:ourgf}, and, since the initial condition belongs to it, the whole trajectory should also belong to it. 
	
	It follows from \eqref{e:linsys} that the DtP map restricted to the straight line \eqref{e:ans} acts as \be \left(U^{(H_1)}(d,d)\right)_i=\frac {c+d}{2d-1}+c, \ i=1,2, \quad |2d|\neq 1. \label{e:rdtp} \ee Leveraging \eqref{envelo}, we observe that the restricted gradient flow \eqref{e:ourgf} reduces to \be \label{e:oded1} d'(\s)=\frac {c+d(\s)}{2d(\s)-1}- \frac {(c+d(\s))^2}{(2d(\s)-1)^2},  \quad d(0)=0,\ee and the trajectory should stay in the ``DtP zone'' \[|2d(\s)|< 1.\] Letting \[\tilde d:=2d-1, \quad \tilde c :=|2c+1|>0,\] we rewrite this ODE in the form \be \label{e:oded} \tilde d'=\frac {\tilde d^2 -\tilde c^2}{2 \tilde d^2},  \quad \tilde d(0)=-1\ee together with the constraint \be \label{e:constode} \quad -2<\tilde d < 0.\ee
	Integrating \eqref{e:oded}, we obtain  \[2\tilde d + \tilde c \ln \left |\frac {\tilde c -\tilde d}{\tilde c +\tilde d}  \right| =\s-2+\tilde c \ln \left |\frac {\tilde c +1}{\tilde c -1}  \right|.\]
	
	If $0<\tilde c <2 $, $\tilde c \neq 1$,  the solution $\tilde d$ satisfies \eqref{e:constode}. Hence,  \[ \tilde c \ln \left |\frac {\tilde c -\tilde d}{\tilde c +\tilde d}  \right| =\s-2+\tilde c \ln \left |\frac {\tilde c +1}{\tilde c -1}  \right|-2\tilde d \to +\infty\] as $\s\to +\infty$, which implies that $\tilde d(\s)$ converges to $\tilde d_\infty:=-\tilde c$ with an explicit exponential rate. The corresponding attractor of \eqref{e:oded1} is \[d_\infty:=\frac 1 2 -\left| c+\frac 1 2\right|.\] The DtP map \eqref{e:rdtp} applied to this limit generates the solution $(c,c)$ if $c<-\frac 1 2$ and $(c+1,c+1)$ if $c>-\frac 1 2$ (i. e., it selects the one that is closer to the base state).  
	
	If $\tilde c=1$ (i.e., $c=0$ or $c=-1$), then the flow is constant: $\tilde d(\s)= -1$. Accordingly, $\tilde{\s}_1^*=0$ and $d (\tilde{\s}_1^*)=0$, and the DtP map generates the solution $(0,0)$. 
	
	If $\tilde c \geq 2$, the trajectory reaches the value $\tilde d=-2$ that corresponds to $d=-\frac 12 $ and lies on the boundary of the ``DtP zone'' at time \be \label{e:sttime} \tilde{\s}^*_1:= -2 + \tilde c \ln \left |\frac {\tilde c +2}{\tilde c -2}  \right| - \tilde c \ln \left |\frac {\tilde c +1}{\tilde c -1}  \right|\ee (with the convention $\tilde{\s}^*_1=+\infty$ for $\tilde c = 2$). After stepping back to the user-defined moment $\s_1^*$ the value of $d(\s_1^*)$ is still approximately equal to $-\frac 1 2$. We change the base state in accordance with \eqref{eq:base_state_change}, and set \be \bar{U}^2 =(v,v):=\left (\frac {c+d(\s_1^*)}{2d(\s_1^*)-1}+c,\frac {c+d(\s_1^*)}{2d(\s_1^*)-1}+c \right)\approx \left (\frac {2c+1}{4},\frac {2c+1}{4} \right).\label{e:vv}\ee 
	
	We run the second stage of our scheme and observe that the trajectories of the second gradient flow \be \frac d {d\s} D^2(\s)=-\nabla S_{H_2}[D^2(\s)], \quad D^2(0)=0, \label{e:ourgf2}\ee are still on the straight line \eqref{e:ans}, for the same reason as before. Hence, the restricted DtP map acts as  \be \left(U^{(H_2)}(d,d)\right)_i=\frac {c+d-v}{2d-1}+c, \ i=1,2, \quad |2d|\neq 1, \label{e:rdtp2} \ee with $v$ defined by \eqref{e:vv}. Consequently, the new gradient flow is
	\be \label{e:oded2} d'(\s)=\frac {c-v+d(\s)}{2d(\s)-1}- \frac {(c-v+d(\s))^2}{(2d(\s)-1)^2},  \quad d(0)=0,\quad |2d(\s)|< 1.\ee Letting \[\tilde d:=2d-1, \quad \tilde c :=|2(c-v)+1|,\] we obtain the ODE \eqref{e:oded}. We will later show that $\tilde c>0$. 
	
	By the same reasoning as above, if \be 0<|2(c-v)+1|<2,\label{e:eqz} \ee including the trivial case $|2(c-v)+1|=1$, $d(\s)$ converges to $d_\infty:=\frac 1 2 -\left| c-v+\frac 1 2\right|$ with an explicit exponential rate (or immediately). The DtP map \eqref{e:rdtp2} generates the solution $(c,c)$ if $c-v<-\frac 1 2$ and $(c+1,c+1)$ if $c-v>-\frac 1 2$.
	
	If $|2(c-v)+1| \ge 2$, the trajectory reaches $d=-\frac 12 $ at the boundary of the ``DtP zone'' at explicit time $\tilde{\s}^*_2$ determined by \eqref{e:sttime} with the updated value of $\tilde c$. 
	
	We repeat the procedure in the same manner for several additional steps until we reach the equilibration zone \eqref{e:eqz} and obtain a solution by the DtP map. In what follows, we show that we only need a \textbf{finite number of steps} and that, at every step, the corresponding constants $\tilde c$ are strictly positive. 
	
	We first observe that the base states are defined by the reciprocal relation
	\begin{multline}  \bar{U}^{k+1} =(v_{k+1},v_{k+1}):=\left (\frac {c-v_k+d(\s_k^*)}{2d(\s_k^*)-1}+c,\frac {c-v_k+d(\s_k^*)}{2d(\s_k^*)-1}+c \right)\\ \approx \left (\frac {2(c+v_k)+1}{4},\frac {2(c+v_k)+1}{4} \right), \quad v_1=0.\label{e:vv1}\end{multline} Thus, \[v_k\approx \left(c+\frac 1 2\right)(1-2^{1-k}), \] which obviously satisfies the second inequality in \eqref{e:eqz} for $k$ large enough. It remains to prove that $|2(c-v_k)+1|>0$, i.e., \[c+\frac 1 2 -v_k\neq 0.\]
	
	Assume for definiteness that $c>-\frac 1 2$ (the opposite case is handled in a similar way). Let us prove by induction that \be c+\frac 1 2 -v_k \geq 2^{1-k}\left(c+\frac 1 2\right)>0.\ee The case $k=1$ is trivial, so we can assume that \[ c+\frac 1 2 -v_{k-1} \geq 2^{2-k}\left(c+\frac 1 2\right).\] By construction, $d(\s_{k-1}^*)\in \left(-\frac 1 2 , 0\right)$, so \[v_{k}=\frac {c+\frac 1 2-v_{k-1}+d(\s_{k-1}^*)-\frac 1 2}{2d(\s_{k-1}^*)-1}+c<\frac {2(c+v_{k-1})+1}{4},\] whence \[c+\frac 1 2 -v_k>\frac 1 2 \left(c+\frac 1 2 -v_{k-1} \right)\geq 2^{1-k}\left(c+\frac 1 2\right).\]
	
	\begin{remark} \label{r:c12} We deliberately excluded the case $c=-\frac 1 2$  because, when started from the zero base state --- as we did --- the scheme enters the classical ``Buridan’s donkey'' dilemma and does not evolve. Indeed, the corresponding DtP map \eqref{e:rdtp} is identically zero, and the gradient flow \eqref{e:oded1} becomes \be d'(\s)=\frac 1 4,  \quad d(0)=0.\ee At time $\tilde{\s}^*_1=2$ the trajectory reaches $d=\frac 12 $ at the boundary of the ``DtP zone'', but at any $\s^*_1 < 2$ the base state produced by the DtP map is $\bar{U}^2 =(0,0)$, cf. \eqref{e:vv}, i.e., the new base state coincides with the old one.    However, a minor perturbation of the zero base state creates a bias and resolves the issue, allowing the scheme to converge to one of the solutions, as it did above\footnote{The same issue was discussed in \cite[Sec.~2.2.2]{udk_thesis} from a different perspective (not involving a gradient flow) for the dual problem in the simple example of the intersection of a circle with a vertical line.}. This indicates that, in potential applications of the scheme to PDEs, certain base states should be avoided (e.g., to prevent excessive symmetry). \end{remark}

	\section{A brief description of the gradient flow scheme for the noise-free Nash system}
	\label{sec5}
	The noise-free Nash system in its Burgers-like formulation \eqref{e:w1} formally reads \be \label{e:mainhjv} \partial_t v+\nabla \vv (v \otimes v)=0, \ v(0,x)=v_0(x):=\nabla \psi_*(x),\ee complemented with spatially-periodic boundary conditions. Thus, it fits into our general framework \eqref{eq:gov_cont_mech} (up to the boundary conditions that are now much easier to handle; note also that here we are denoting the unknown function for the primal problem by $v$ instead of $U$ and that $u=d=n=pN^2$). 
	
	We define the auxiliary potential $H$ as in \eqref{auxp}:\be \label{auxp2}H(v)=\frac 1 2 |v-\vbs|^2.\ee Then, denoting the dual variable by $a$, we observe that \begin{equation}
	\label{eq:predualn}
	\widehat{S}_H[v,a]=- \int_0^T \left[\ren(t)+(v-v_0,E)+(v\otimes v, B)\right]\, dt,   
	\end{equation} where $B=\vv^* \di a$ and $E=\p_t a$ as in Section \ref{sec21}. Accordingly, the ``DtP zone'' is determined by the condition \be \label{e:goodna}I+2B>0\ \mathrm{a.e.}\ \mathrm{in}\ (0,T)\times \Omega,\ee and the DtP map acts as in Theorem \ref{t:smooth}: \be \label{e:dualsolinv2} U^{(H)}(E,B)=(I+2B)^{-1}(\vbs-E).\ee (\dv{Here it is appropriate to use the pair $(E,B)$ instead of the full vector $\dee=(\p_t a, \nabla a, a)$, because the pre-dual functional \eqref{eq:predualn} depends on $\dee$ only through $(E,B)$}). Hence, \begin{multline*}S_H[a; \bar{v}] = \widehat{S}_H \left[ U^{(H)}(E,B,\bar{v}), a \right]=-\cc \\ -\int_0^T \left[(U^{(H)}(E,B,\bar{v})-v_0,E-\vbs)+\frac 1 2\left((U^{(H)}(E,B,\bar{v}))\otimes (U^{(H)}(E,B,\bar{v})), I+2B\right)\right]\, dt\\=\int_0^T \left[(v_0,E)-\frac 1 2 (\vbs,\vbs)+\frac 1 2\left((E-\vbs)\otimes (E-\vbs), (I+2B)^{-1}\right)\right]\, dt. \end{multline*} 
	This is a restriction  of the convex  functional \[\widetilde{S}_H [a]  := \sup_U \widehat{S}_H[v,a]=\int_0^T \left[(v_0,E)-\frac 1 2 (\vbs,\vbs)\right]\, dt-\bigk\] to the ``DtP zone'' \eqref{e:goodna}.
	
	The minimizers of $\widetilde{S}_H$   exist by Theorem \ref{t:ex} (but may lie outside of the ``DtP zone'') and belong to the class \eqref{e:be} (in particular, the corresponding $a=\int_T^t E\in \mathcal H$).
	
	Consequently, 
	\be
	\int_0^T \left(\frac{\delta S_H}{\delta a}\big[a; \bar{v} \big] , \delta a\right)\, dt \\=-\int_0^T \left[(v-v_0,\delta E)+(v\otimes v, \delta B)\right]\, dt
	\ee
	with $v=(I+2B)^{-1}(\vbs-E)$.
	
	For the first stage of the gradient flow scheme, we select an arbitrary first base state $\vbs^1(t,x)$.  For example, we can take $\vbs^1(t,x)=v_0(x)$. We run the first gradient descent in the Hilbert space $\mathcal H$ for the unknown vector function $a^1(t,x,\s)$:
	\begin{multline}
	\int_0^T \left(\p_\s a^1, \delta a\right)\, dt \\=\int_0^T \left[(v^1-v_0,\delta E)+(v^1\otimes v^1, \delta B)\right]\, dt,\quad a^1(T,x,\s) = 0, \quad a^1(t,x,0) = 0, \label{e.56}
	\end{multline}
	where \be v^1(t,x,\s):=(I+2 \vv^* \di a^1(t,x,\s))^{-1}(\vbs^1(t,x)-\p_t a^1(t,x,\s)). \label{e.57}\ee
	Here the variation  $\delta a(t,x)$ should satisfy \[\delta a(T,x) = 0.\]
	If we test \eqref{e.56} with compactly supported $\delta a$ and integrate by parts, we formally get the following nonlinear degenerate parabolic PDE system of $n$ equations: \be \p_\s a^1= -\p_t v^1-\nabla \vv (v^1\otimes v^1), \ee with $v^1$ defined by \eqref{e.57}. Testing now with arbitrary $\delta a$, we infer that the boundary (and initial) conditions are \be v^1(0,x,\s)=v_0(x),\quad a^1(T,x,\s) = 0, \quad a^1(t,x,0) = 0. \label{e.58}
	\ee

	The process is repeated until stage $k_*$ at which we reach the equilibration zone or continued for infinite number of stages. In the first case, the corresponding $$v(t,x):=v^{k_*}(t,x,\tilde{\s}^*_{k_*})$$ would allegedly be a solution to the Nash system \eqref{e:w1}. In the second case, we generate an infinite sequence of base states \[\vbs^k(t,x):=v^{k-1}(t,x,{\s}^*_{k-1}).\] If this sequence has accumulation points in some reasonable topology (say, in the weak $L^2$-topology), those points can be referred to as ``generalized'' solutions to the Nash system.

	\subsection*{Acknowledgement} 
	This work originated from discussions held at Carnegie Mellon University while DV was visiting AA; DV acknowledges the hospitality of CMU. DV's work was supported by CMUC\footnote{https://doi.org/10.54499/UID/00324/2025} under FCT, grants UID/00324/2025 and UID/PRR/00324/2025, and by KAUST Research Funding (KRF), award ORFS-2024-CRG12-6430.3. AA  thanks Marshall Slemrod for suggesting the Nash system for study.

\end{document}